\documentclass[12pt, a4paper]{article}
\usepackage[T1]{fontenc}
\usepackage{lmodern}
\usepackage{amsmath}
\usepackage{amsthm}
\usepackage{amssymb}
\usepackage{graphicx}
\usepackage{url}
\usepackage{here}
\theoremstyle{plain}
\newtheorem{theorem}{Theorem}[section]

\newtheorem{lemma}[theorem]{Lemma}

\newtheorem{conjecture}[theorem]{Conjecture}

\newtheorem{definition}[theorem]{Definition}
\newtheorem{remark}[theorem]{Remark}
\newtheorem{claim}{Claim}

\newcommand{\fig}[5]{
\begin{figure}[H]
\begin{center}
\resizebox{#1}{#2}{\includegraphics{#3}}
\end{center}
\caption{#4}
\label{#5}
\end{figure}
}
\newcommand{\C}{\mathbb{C}}
\title{
$(g,f)$-Chromatic spanning trees and forests
\footnotetext{
This work was supported by JSPS KAKENHI Grant Number 16K05254.
}
}
\author{
Kazuhiro Suzuki%
\footnote{
Department of Information Science,
Kochi University,
Japan.
kazuhiro@tutetuti.jp.
}
}
\date{\empty}

\begin{document}
%------------------------------------
\maketitle

%------------------------------------
%   Abstract
%------------------------------------
\begin{abstract}
A heterochromatic (or rainbow) graph
is an edge-colored graph
whose edges have distinct colors,
that is,
where each color appears at most once.
In this paper,
I propose a $(g,f)$-chromatic graph
as an edge-colored graph
where each color $c$ appears
at least $g(c)$ times and at most $f(c)$ times.
I also present a necessary and sufficient condition
for edge-colored graphs (not necessary to be proper)
to have a $(g,f)$-chromatic spanning tree.
Using this criterion,
I show that
an edge-colored complete graph $G$ has
a spanning tree
with a color probability distribution
``similar'' to that of $G$.
Moreover, I conjecture that
an edge-colored complete graph $G$
of order $2n$ $(n \ge 3)$
can be partitioned into
$n$ edge-disjoint spanning trees
such that
each has a color probability distribution
``similar'' to that of $G$.
\\[6pt]
\end{abstract}

%------------------------------------
%   Section 1
\section{Introduction}
%------------------------------------
We consider finite undirected graphs
without loops or multiple edges.
For a graph $G$, we denote by $V(G)$ and $E(G)$
its vertex and edge sets, respectively.
An \textit{edge-coloring} of a graph $G$
is a mapping $color:E(G) \rightarrow \C$,
where $\C$ is a set of colors.
Then, the triple $(G, \C, color)$
is called an \textit{edge-colored graph}.
We often abbreviate an edge-colored graph $(G, \C, color)$ as $G$.
Note that
an edge colored graph is not necessary to be proper,
where
distinct red edges may have a common end vertex.

%------------------------------------
\subsection{%
Heterochromatic (or rainbow) spanning trees}
%------------------------------------
An edge-colored graph $G$ is said to be \textit{heterochromatic}%
\footnote{%
A heterochromatic graph is also said to be 
\textit{rainbow}, \textit{multicolored}, \textit{totally multicolored},
\textit{polychromatic}, or \textit{colorful},
and so on.
}
if no two edges of $G$ have the same color,
that is,
$color(e_i) \ne color(e_j)$
for any two distinct edges $e_i$ and $e_j$ of $G$.
As far as I know,
there are three topics about heterochromatic graphs:
the Anti-Ramsey problem
introduced by Erd\H{o}s et al.
\cite{Erdos1975Anti-RamseyTheorems},
rainbow connection problems
introduced by Chartrand et al.
\cite{Chartrand2008RAINBOWGRAPHS},
and heterochromatic subgraph problems,
(see the surveys
\cite{Fujita2010RainbowSurvey}
\cite{Li2013RainbowSurvey}
\cite{Kano2008MonochromaticSurvey}
).
This paper focuses
on heterochromatic subgraph problems.

We denote by $\omega(G)$ the number of components of a graph $G$.
Given an edge-colored graph $G$ and a color set $R$,
we define $E_{R}(G) = \{ e \in E(G) ~|~ color(e) \in R \}$.
For simplicity,
we denote the graph
$(V(G), E(G) \setminus E_R(G))$ by $G-E_R(G)$,
and also denote
$E_{\{c\}}(G)$ by $E_c(G)$ for a color $c$.

Akbari \& Alipour
\cite{Akbari2007MulticoloredGraphs} 
and Suzuki
\cite{Suzuki2006AGraph}
independently presented
a necessary and sufficient condition
for edge-colored graphs
to have a heterochromatic spanning tree.

\begin{theorem}[%
Akbari and Alipour
\cite{Akbari2007MulticoloredGraphs},
Suzuki
\cite{Suzuki2006AGraph}]
An edge-colored graph $G$
has a heterochromatic spanning tree
if and only if
\begin{equation*}
\omega(G-E_R(G)) \le |R|+1
\text{~~~~~ for any } R \subseteq \C.
\end{equation*}
\label{thm_Suzuki2006AGraph_1}
\end{theorem}

Suzuki
\cite{Suzuki2006AGraph}
proved the following theorem by using Theorem
\ref{thm_Suzuki2006AGraph_1}.

\begin{theorem}[Suzuki \cite{Suzuki2006AGraph}]
An edge-colored complete graph $G$ of order n
has a heterochromatic spanning tree
if $|E_c(G)| \le n/2$ for any color $c \in \C$.
\label{thm_Suzuki2006AGraph_2}
\end{theorem}

The complete graph $K_n$ has $(n-1)n/2$ edges,
thus the condition of Theorem
\ref{thm_Suzuki2006AGraph_2}
is equivalent to that
\begin{equation*}
\frac{|E_c(G)|}{|E(G)|}(n-1) \le 1
\text{~~~~~ for any color } c \in \C.
\end{equation*}

We can regard $|E_c(G)|/|E(G)|$ as
the probability of a color $c$
appearing in $G$.
The term ``Heterochromatic'' means that
any color appears once or zero times.
Thus,
we can interpret Theorem \ref{thm_Suzuki2006AGraph_2}
as saying that
if each color probability is at most $1/(n-1)$ in $G$
then $G$ has a spanning tree $T$
such that
each color probability is $1/(n-1)$ or $0$ in $T$.

%------------------------------------
\subsection{$f$-Chromatic spanning trees and forests}
%------------------------------------
The term ``Heterochromatic'' means that
any color appears at most once.
Suzuki
\cite{Suzuki2013AForests}
generalized ``once'' to a mapping $f$
from a given color set $\C$
to the set $\mathbb{Z}_{\ge 0}$
of non-negative integers,
and defined \textit{$f$-chromatic} graphs
as follows.

\begin{definition}[Suzuki \cite{Suzuki2013AForests}]
Let $G$ be an edge-colored graph.
Let $f$ be a mapping from $\C$
to $\mathbb{Z}_{\ge 0}$.
$G$ is said to be \textit{$f$-chromatic}
if $|E_c(G)| \le f(c)$ for any color $c \in \C$. 
\label{def_Suzuki2013AForests}
\end{definition}

Fig. \ref{fig_180825_1} shows an example
of an $f$-chromatic spanning tree
of an edge-colored graph.
For the color set $\C = \{1,2,3,4,5,6,7 \}$,
a mapping $f$ is given as follows:
\begin{gather*}
f(1)=3,
f(2)=2,
f(3)=3,
f(4)=0,
f(5)=0,
f(6)=1,
f(7)=2.
\end{gather*}
The left edge-colored graph
has the right $f$-chromatic spanning tree,
where each color $c$ appears at most $f(c)$ times.

\fig{0.85\textwidth}{!}
{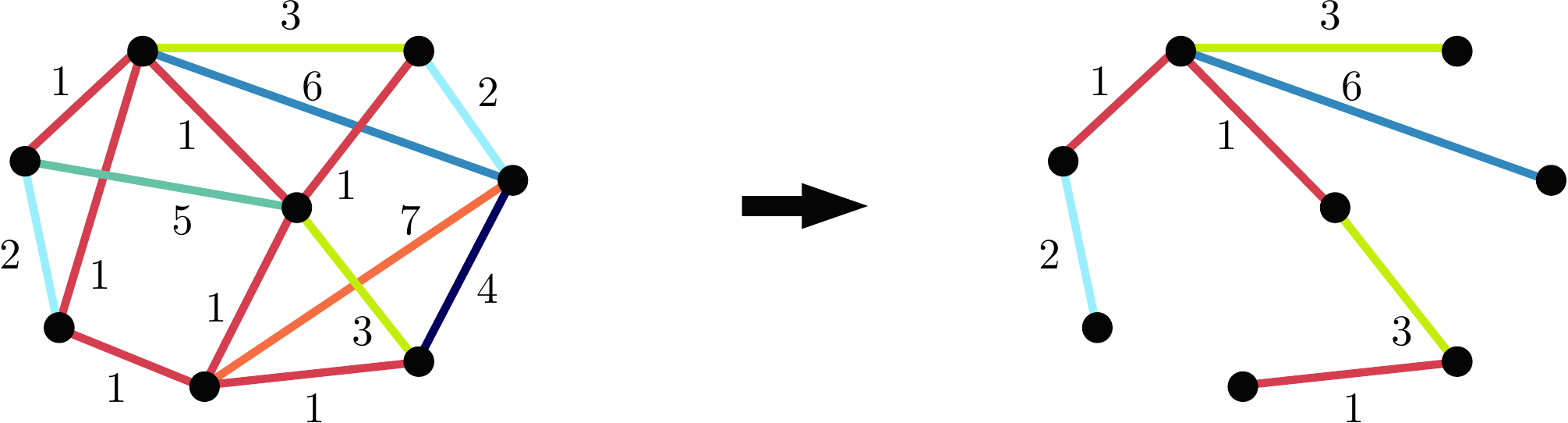}
{An $f$-chromatic spanning tree
of an edge-colored graph.}
{fig_180825_1}

Suzuki
\cite{Suzuki2013AForests}
presented
the following necessary and sufficient condition
for edge-colored graphs to have
an $f$-chromatic spanning forest
with exactly $m$ components.

\begin{theorem}[Suzuki \cite{Suzuki2013AForests}]
Let $G$ be an edge-colored graph of order $n$.
Let $f$ be a mapping from $\C$
to $\mathbb{Z}_{\ge 0}$.
Let $m$ be a positive integer such that $n \ge m$.
$G$ has
an $f$-chromatic spanning forest
with exactly $m$ components
if and only if
\begin{equation*}
\omega(G-E_R(G)) \le m+\sum_{c \in R}f(c)
\text{~~~~~ for any } R \subseteq \C.
\end{equation*}
\label{thm_Suzuki2013AForests_1}
\end{theorem}

Suzuki
\cite{Suzuki2013AForests}
proved the following Theorem
by using Theorem
\ref{thm_Suzuki2013AForests_1}.

\begin{theorem}[Suzuki \cite{Suzuki2013AForests}]
Let $G$ be an edge-colored graph of order $n$.
Let $f$ be a mapping from $\C$
to $\mathbb{Z}_{\ge 0}$.
Let $m$ be a positive integer such that $n \ge m$.
If $|E(G)|>\binom{n-m}{2}$ and
\begin{equation*}
\frac{|E_c(G)|}{|E(G)|}(n-m) \le f(c)
\text{~~~~~ for any color } c \in \C,
\end{equation*}
then $G$ has an $f$-chromatic spanning forest
with exactly $m$ components.
\label{thm_Suzuki2013AForests_2}
\end{theorem}

A heterochromatic graph is
an $f$-chromatic graph with $f(c) = 1$
for every color $c$.
Thus, these two theorems include
Theorem
\ref{thm_Suzuki2006AGraph_1} and
Theorem
\ref{thm_Suzuki2006AGraph_2}.
In this paper,
we will further generalize these theorems
and study color probability distributions
of edge-colored complete graphs
and its spanning trees.

%------------------------------------
%   Section 2
\section{Main results}
%------------------------------------
In this paper,
I propose a \textit{$(g,f)$-chromatic graph}
as an edge-colored graph
where each color $c$ appears
at least $g(c)$ times and at most $f(c)$ times.
I also present a necessary and sufficient condition
for edge-colored graphs
to have a $(g,f)$-chromatic spanning forest
with exactly $m$ components.
(Theorem \ref{thm_180815_1}).
Using this criterion,
I show that
an edge-colored complete graph $G$ has
a spanning tree
with a color probability distribution
``similar'' to that of $G$
(Theorem \ref{thm_180913_1}).
Moreover, I conjecture that
an edge-colored complete graph $G$
of order $2n$ $(n \ge 3)$
can be partitioned into
$n$ edge-disjoint spanning trees
such that each has a color probability distribution
``similar'' to that of $G$
(Conjecture \ref{conj_180910_1}).

%------------------------------------
\subsection{%
$(g,f)$-Chromatic spanning trees and forests}
%------------------------------------
We begin with the definition
of a \textit{$(g,f)$-chromatic} graph.

\begin{definition}
Let $G$ be an edge-colored graph.
Let $g$ and $f$ be mappings from $\C$
to $\mathbb{Z}_{\ge 0}$.
$G$ is said to be \textit{$(g,f)$-chromatic}
if $g(c) \le |E_c(G)| \le f(c)$
for any color $c \in \C$. 
\label{def_180815_1}
\end{definition}

Fig. \ref{fig_180825_2} shows
a $(g,f)$-chromatic spanning tree
of an edge-colored graph.
For the color set $\C = \{1,2,3,4,5,6,7 \}$,
mappings $g$ and $f$ are given as follows:
\begin{gather*}
g(1)=1,
g(2)=1,
g(3)=2,
g(4)=0,
g(5)=0,
g(6)=1,
g(7)=0,\\
f(1)=3,
f(2)=2,
f(3)=3,
f(4)=0,
f(5)=0,
f(6)=1,
f(7)=2.
\end{gather*}
The left edge-colored graph
has the right $(g,f)$-chromatic spanning tree,
where each color $c$ appears
at least $g(c)$ times and at most $f(c)$ times.

\fig{0.85\textwidth}{!}
{fig_180825_1.pdf}
{A $(g,f)$-chromatic spanning tree
of an edge-colored graph.}
{fig_180825_2}

We will see more examples.
First, we suppose that
$g$ and $f$ are given as follows:
\begin{gather*}
g(1)=3,
g(2)=1,
g(3)=3,
g(4)=0,
g(5)=0,
g(6)=1,
g(7)=2,\\
f(1)=3,
f(2)=2,
f(3)=3,
f(4)=0,
f(5)=0,
f(6)=1,
f(7)=2.
\end{gather*}
Then, the left edge-colored graph
in Fig. \ref{fig_180825_2}
has no $(g,f)$-chromatic spanning trees,
because $g(1)+g(2)+ \cdots +g(7)$ exceeds $7$,
the size of a spanning tree of the graph.

Next, in Fig. \ref{fig_180826_1},
we suppose that
$g$ and $f$ are given as follows:
\begin{gather*}
g(1)=0,
g(2)=2,
g(3)=2,
g(4)=0,
g(5)=0,
g(6)=1,
g(7)=0,\\
f(1)=3,
f(2)=2,
f(3)=3,
f(4)=0,
f(5)=0,
f(6)=1,
f(7)=2.
\end{gather*}
Then, in the left edge-colored graph,
any subgraph having $g(2)$, $g(3)$, and $g(6)$ edges
colored with $2$, $3$, and $6$, respectively,
contains the right subgraph,
which has a cycle.
Thus, the left graph
has no $(g,f)$-chromatic spanning trees.

\fig{0.85\textwidth}{!}
{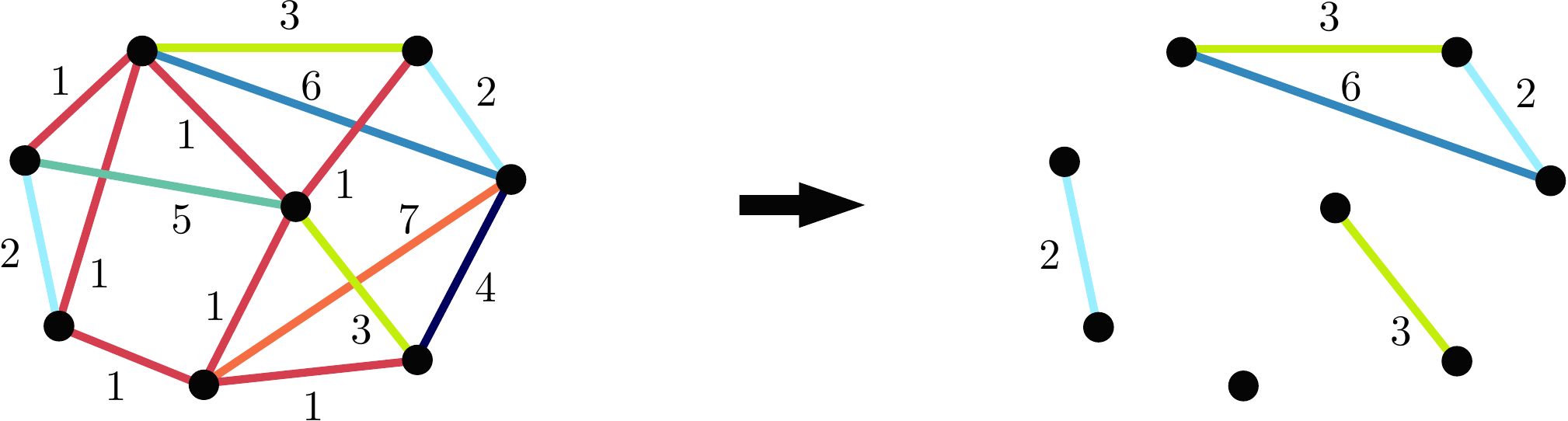}
{The mapping $g$ forces us to use a cycle.}
{fig_180826_1}

The following is the main theorem,
which gives a necessary and sufficient condition
for edge-colored graphs
to have a $(g,f)$-chromatic spanning tree
as a corollary.

\begin{theorem}
Let $G$ be an edge-colored graph of order $n$.
Let $g$ and $f$ be mappings from $\C$
to $\mathbb{Z}_{\ge 0}$
such that $g(c) \le f(c)$ for any $c \in \C$.
Let $m$ be a positive integer
such that $n \ge m + \sum_{c \in \C}g(c)$.
$G$ has
a $(g,f)$-chromatic spanning forest
with exactly $m$ components
if and only if
\begin{equation*}
\omega(G-E_R(G)) \le
\min
\{~
m+\sum_{c \in R}f(c),~~
n-\sum_{c \in \C \setminus R}g(c)
~\}
\text{~~~ for any } R \subseteq \C.
\end{equation*}
\label{thm_180815_1}
\end{theorem}

This theorem is proved in Section \ref{proof22}.
Note that
the size of a spanning forest
with exactly $m$ components of $G$ is $n-m$.
If $G$ has a $(g,f)$-chromatic spanning forest
with exactly $m$ components,
then the size of the forest
is at least $\sum_{c \in \C}g(c)$.
Thus,
the condition $n \ge m + \sum_{c \in R}g(c)$
is necessary.

We see the above last example again.
Let $G$ be the left graph in
Fig. \ref{fig_180826_1}.
$G$ has no $(g,f)$-chromatic spanning trees.
Thus, by Theorem \ref{thm_180815_1},
\begin{equation*}
\omega(G-E_R(G)) >
\min
\{~
1+\sum_{c \in R}f(c),~~
8-\sum_{c \in \C \setminus R}g(c)
~\}
\text{~~~ for some } R \subseteq \C.
\end{equation*}
Actually,
for $R = \{ 1,4,5,7 \}$,
$G-E_R(G)$ is the right graph in
Fig. \ref{fig_180826_1}
and we have
\begin{equation*}
\omega(G-E_R(G)) = 4,~~
1+\sum_{c \in R}f(c) = 6,~~
8-\sum_{c \in \C \setminus R}g(c) = 3.
\end{equation*}

We can prove the following theorem
by using Theorem \ref{thm_180815_1}.

\begin{theorem}
Let $G$ be an edge-colored graph of order $n$.
Let $g$ and $f$ be mappings from $\C$
to $\mathbb{Z}_{\ge 0}$.
Let $m$ be a positive integer such that $n \ge m$.
If $|E(G)|>\binom{n-m}{2}$ and
\begin{equation*}
g(c) \le \frac{|E_c(G)|}{|E(G)|}(n-m) \le f(c)
\text{~~~~~ for any color } c \in \C,
\end{equation*}
then $G$ has a $(g,f)$-chromatic spanning forest
with exactly $m$ components.
\label{thm_180826_1}
\end{theorem}

This theorem is proved in Section \ref{proof23}.
Note that
an $f$-chromatic graph is
a $(g,f)$-chromatic graph with $g(c) = 0$
for any color $c$,
and $\omega(G-E_R(G)) \le n$
for any $R \subseteq \C$
since the number of components of any subgraph
of a graph of order $n$ is at most $n$.
Thus,
Theorem
\ref{thm_180815_1}
and
\ref{thm_180826_1}
include
Theorem
\ref{thm_Suzuki2013AForests_1}
and
\ref{thm_Suzuki2013AForests_2}.

%------------------------------------
\subsection{%
Color probability distributions
of edge-colored graphs}
\label{sec_Cpd}
%------------------------------------

We call $|E_c(G)|/|E(G)|$
the \textit{color probability}
of a color $c$
in an edge-colored graph $G$.
The \textit{color probability distribution} of $G$
is the sequence of the color probabilities.
Does a given edge-colored complete graph $G$ have
a spanning tree
with the same color probability distribution
as that of $G$ ?
Fig. \ref{fig_180830_1} shows an example.

\fig{0.75\textwidth}{!}
{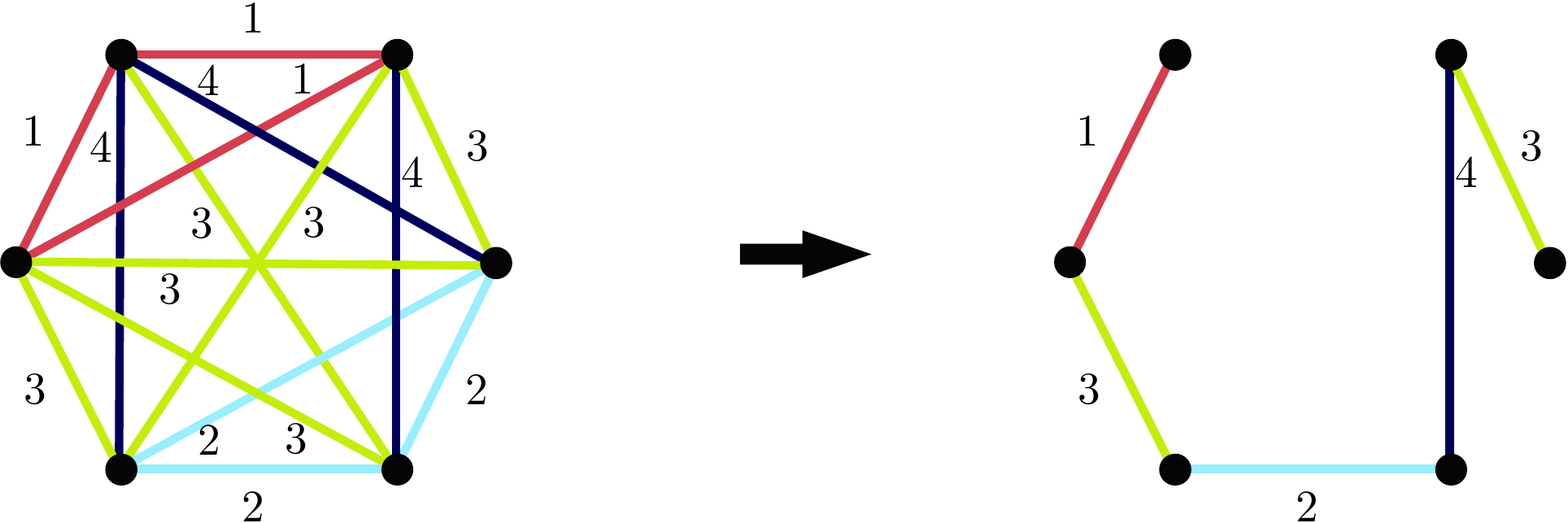}
{An edge-colored complete graph
having a spanning tree
with the same color probability distribution
as that of it.}
{fig_180830_1}

Let $G$ and $T$ be the left and right graph
in Fig. \ref{fig_180830_1}, respectively.
Then,
\begin{gather*}
|E_1(G)|=3, |E_2(G)|=3, |E_3(G)|=6, |E_4(G)|=3,
|E(G)|=15,\\
|E_1(T)|=1, |E_2(T)|=1, |E_3(T)|=2, |E_4(T)|=1,
|E(T)|=5.
\end{gather*}
Thus, both color probability distributions are
$(0.2, 0.2, 0.4, 0.2)$
(Fig. \ref{fig_180830_2}).

\fig{0.75\textwidth}{!}
{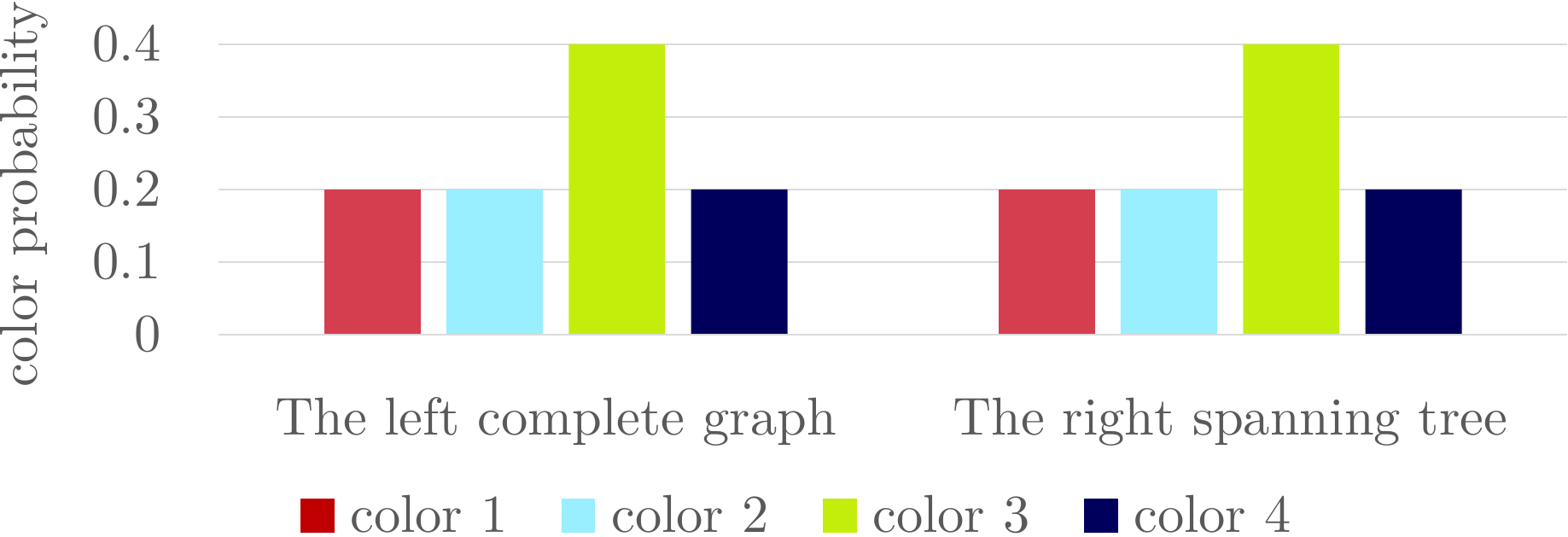}
{Color probability distributions
of the graphs in Fig. \ref{fig_180830_1}.}
{fig_180830_2}

Fig. \ref{fig_180830_3} shows another example.
In the left complete graph $G$ of order $n=6$
has no spanning trees
with the same color probability distribution
as that of $G$,
because the number of colors in $G$ is $7$
exceeding the size of a spanning tree of $G$.

\fig{0.75\textwidth}{!}
{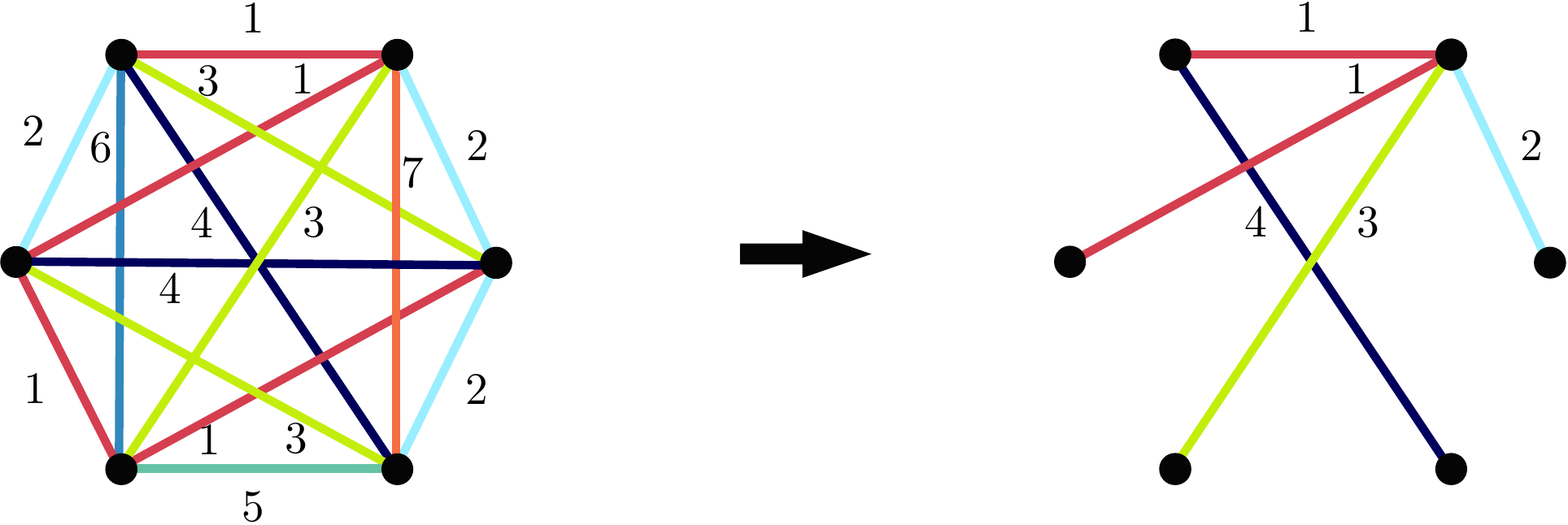}
{An edge-colored complete graph
having no spanning trees
with the same color probability distribution
as that of it,
but having a spanning tree
with a color probability distribution
similar to that of it.}
{fig_180830_3}

In other words,
for any spanning tree $T$ of $G$,
\begin{equation*}
\frac{|E_c(G)|}{|E(G)|} \ne \frac{|E_c(T)|}{|E(T)|}
\text{~~ for some color } c \in \C,
\end{equation*}
that is,
\begin{equation*}
|E_c(T)| \ne \frac{|E_c(G)|}{|E(G)|}|E(T)|
= \frac{|E_c(G)|}{|E(G)|}(n-1)
\text{~~ for some color } c \in \C.
\end{equation*}

However,
the right spanning tree has
a color probability distribution
\textit{similar} to that of $G$
(see Fig. \ref{fig_180830_4}),
in the sense that
the following condition holds:
\begin{equation*}
\left\lfloor
\frac{|E_c(G)|}{|E(G)|}(n-1)
\right\rfloor
\le
|E_c(T)|
\le
\left\lceil
\frac{|E_c(G)|}{|E(G)|}(n-1)
\right\rceil
\text{~~ for any color } c \in \C.
\end{equation*}

\fig{0.85\textwidth}{!}
{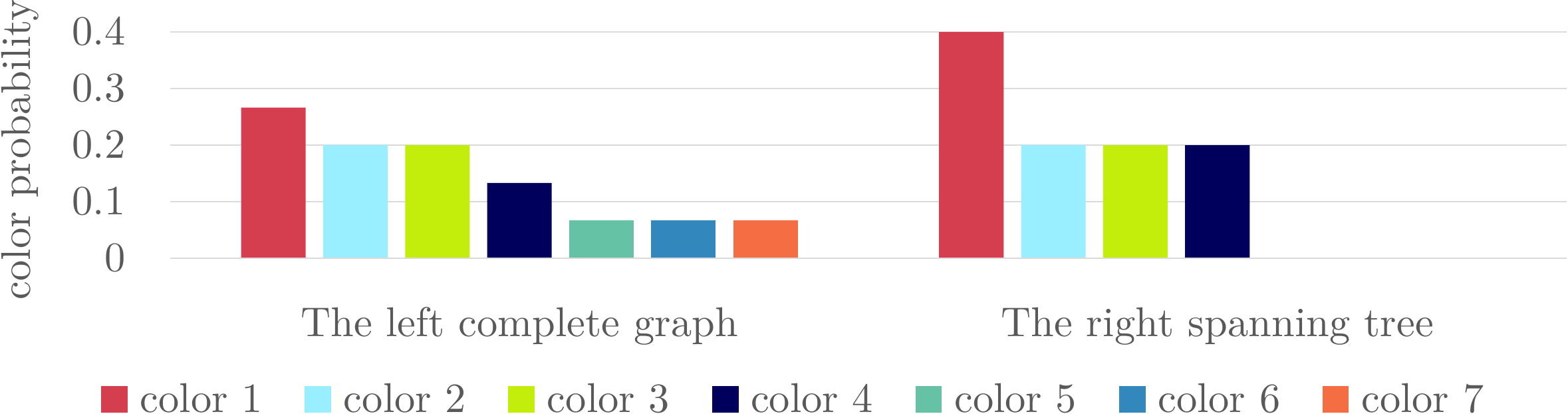}
{Color probability distributions
of the graphs in Fig. \ref{fig_180830_3}.}
{fig_180830_4}

In general,
let $G$ be an edge-colored complete graph
of order $n$,
and set
\begin{equation*}
g(c) =
\left\lfloor
\frac{|E_c(G)|}{|E(G)|}(n-1)
\right\rfloor
\text{~~and~~}
f(c) =
\left\lceil
\frac{|E_c(G)|}{|E(G)|}(n-1)
\right\rceil
\text{~~ for any color } c \in \C.
\end{equation*}
Then, by Theorem \ref{thm_180826_1},
$G$ has a $(g,f)$-chromatic spanning tree $T$.
By the definition \ref{def_180815_1},
$g(c) \le |E_c(T)| \le f(c)$
for any color $c \in \C$.
Thus, the following theorem holds.

\begin{theorem}
Any edge-colored complete graph $G$ of order $n$
has a spanning tree
with a color probability distribution
similar to that of $G$,
that is,
$G$ has a spanning tree $T$
such that
\begin{equation*}
\left\lfloor
\frac{|E_c(G)|}{|E(G)|}(n-1)
\right\rfloor
\le
|E_c(T)|
\le
\left\lceil
\frac{|E_c(G)|}{|E(G)|}(n-1)
\right\rceil
\text{~~ for any color } c \in \C.
\end{equation*}
\label{thm_180913_1}
\end{theorem}

\begin{remark}
Theorem \ref{thm_180913_1} is equivalent
to Theorem \ref{thm_180826_1} with $m=1$
for any edge-colored complete graph $G$
of order $n$.
\label{rem_180922_1}
\end{remark}
\begin{proof}
Theorem \ref{thm_180913_1} follows
from Theorem \ref{thm_180826_1}
by the above argument.

Let $G$ be an edge-colored graph of order $n$.
Let $g$ and $f$ be mappings from $\C$
to $\mathbb{Z}_{\ge 0}$.
Suppose that Theorem \ref{thm_180913_1} holds.
Then,
$G$ has a spanning tree $T$
with a color probability distribution
similar to that of $G$.

If $g$,$f$, and $G$ satisfy the condition in Theorem \ref{thm_180826_1},
then we have
\begin{equation*}
g(c) \le
\left\lfloor
\frac{|E_c(G)|}{|E(G)|}(n-1)
\right\rfloor
\text{ and }
\left\lceil
\frac{|E_c(G)|}{|E(G)|}(n-1)
\right\rceil
\le
f(c)
\text{~~ for any color } c \in \C.
\end{equation*}

Hence,
$T$ satisfies that for any color $c \in \C$,
\begin{equation*}
g(c) \le
\left\lfloor
\frac{|E_c(G)|}{|E(G)|}(n-1)
\right\rfloor
\le
|E_c(T)|
\le
\left\lceil
\frac{|E_c(G)|}{|E(G)|}(n-1)
\right\rceil
\le f(c).
\end{equation*}

Therefore,
$T$ is a $(g,f)$-chromatic spanning tree of $G$.
\end{proof}

From Theorem \ref{thm_180913_1},
we can get the following theorem,
proved in Section \ref{proof26}.

\begin{theorem}
An edge-colored complete graph $G$ of order $n$
has a spanning tree
with the same color probability distribution
as that of $G$
if and only if
$|E_c(G)|$ is an integral multiple of $n/2$
for any color $c \in \C$.
\label{thm_180830_1}
\end{theorem}

%------------------------------------
\subsection{%
Spanning tree
decomposition conjectures}
%------------------------------------
In 1996,
Brualdi and Hollingsworth
\cite{Brualdi1996MulticoloredGraphs}
presented the following conjecture.

\begin{conjecture}[Brualdi and Hollingsworth
\cite{Brualdi1996MulticoloredGraphs}]
A properly edge-colored complete graph
$K_{2n}$ $(n \ge 3)$
with exactly $2n-1$ colors
can be partitioned into
$n$ edge-disjoint heterochromatic spanning trees.
\label{conj_Brualdi1996MulticoloredGraphs_1}
\end{conjecture}

Brualdi and Hollingsworth
\cite{Brualdi1996MulticoloredGraphs}
proved that
a properly edge-colored complete graph
$K_{2n} (n \ge 3)$
with exactly $2n-1$ colors
has two edge-disjoint heterochromatic spanning trees.
Krussel, Marshall, and Verrall
\cite{Krussel2000SpanningK2n}
proved that the graph
has three edge-disjoint heterochromatic spanning trees.
Kaneko, Kano, and Suzuki
\cite{Kaneko2006TwoGraphs}
proved that
a properly edge-colored complete graph
$K_n (n \ge 5)$
(not necessary with exactly $n-1$ colors)
has two edge-disjoint heterochromatic spanning trees.
Akbari and Alipour
\cite{Akbari2007MulticoloredGraphs}
proved that
an edge-colored complete graph $G$
(not necessary to be proper)
of order $n ( \ge 5)$
has two edge-disjoint heterochromatic spanning trees
if $|E_c(G)| \le n/2$ for any color $c \in \C$.
Carraher, Hartke, and Horn
\cite{Carraher2016Edge-disjointGraphs}
proved that
an edge-colored complete graph $G$
of order $n ( \ge 1000000)$
has at least $\lfloor n/(1000 \log n) \rfloor$
edge-disjoint heterochromatic spanning trees
if $|E_c(G)| \le n/2$ for any color $c \in \C$.
Horn
\cite{Horn2018RainbowOne-factorizations}
proved that
there exist positive constants $\epsilon, n_0$
so that
every properly edge-colored complete graph
$K_{2n} (2n \ge n_0)$
with exactly $2n-1$ colors
has at least $\epsilon n$ edge-disjoint
heterochromatic spanning trees.

Based on these previous results,
I conjecture the following
as a generalization of Conjecture
\ref{conj_Brualdi1996MulticoloredGraphs_1}.

\begin{conjecture}
An edge-colored complete graph $G$
of order $2n$ $(n \ge 3)$
can be partitioned into
$n$ edge-disjoint spanning trees
$T_1, T_2, \ldots, T_n$
such that each has
a color probability distribution
is similar to that of $G$,
that is,
each $T_i$ satisfies that
\begin{equation*}
\left\lfloor
\frac{|E_c(G)|}{|E(G)|}(2n-1)
\right\rfloor
\le
|E_c(T_i)|
\le
\left\lceil
\frac{|E_c(G)|}{|E(G)|}(2n-1)
\right\rceil
\text{~~ for any color } c \in \C.
\end{equation*}
\label{conj_180910_1}
\end{conjecture}

Since $|E(G)|=2n(2n-1)/2$
for the complete graph $G$ of order $2n$,
we have
\begin{equation*}
\frac{|E_c(G)|}{|E(G)|}(2n-1)
= \frac{|E_c(G)|}{n}.
\end{equation*}
Thus, this conjecture implies that
$E_c(G)$ can be partitioned into
$n$ almost equal parts.
Fig. \ref{fig_180910_1}
shows an edge-colored complete graph $G$ of order $6$
and its partition into three edge-disjoint
spanning trees $T_1$, $T_2$, and $T_3$.
In this example,
\begin{gather*}
|V(G)|=2n=6, |E(G)|=15,\\
|E_1(G)|=7, |E_2(G)|=4,
|E_3(G)|=2, |E_4(G)|=2,\\
|E_1(T_1)|=3, |E_2(T_1)|=1,
|E_3(T_1)|=1, |E_4(T_1)|=0,\\
|E_1(T_2)|=2, |E_2(T_2)|=2,
|E_3(T_2)|=0, |E_4(T_2)|=1,\\
|E_1(T_3)|=2, |E_2(T_3)|=1,
|E_3(T_3)|=1, |E_4(T_3)|=1.
\end{gather*}
Thus, $E_c(G)$ is partitioned
into three almost equal parts for each color $c$,
and each $T_i$ $(1 \le i \le 3)$ satisfies that
\begin{equation*}
\left\lfloor
\frac{|E_c(G)|}{|E(G)|}(2n-1)
\right\rfloor
\le
|E_c(T_i)|
\le
\left\lceil
\frac{|E_c(G)|}{|E(G)|}(2n-1)
\right\rceil
\text{~~ for any color } c \in \C.
\end{equation*}
Hence, each $T_i$
has a color probability distribution
similar to that of $G$.

\fig{1.0\textwidth}{!}
{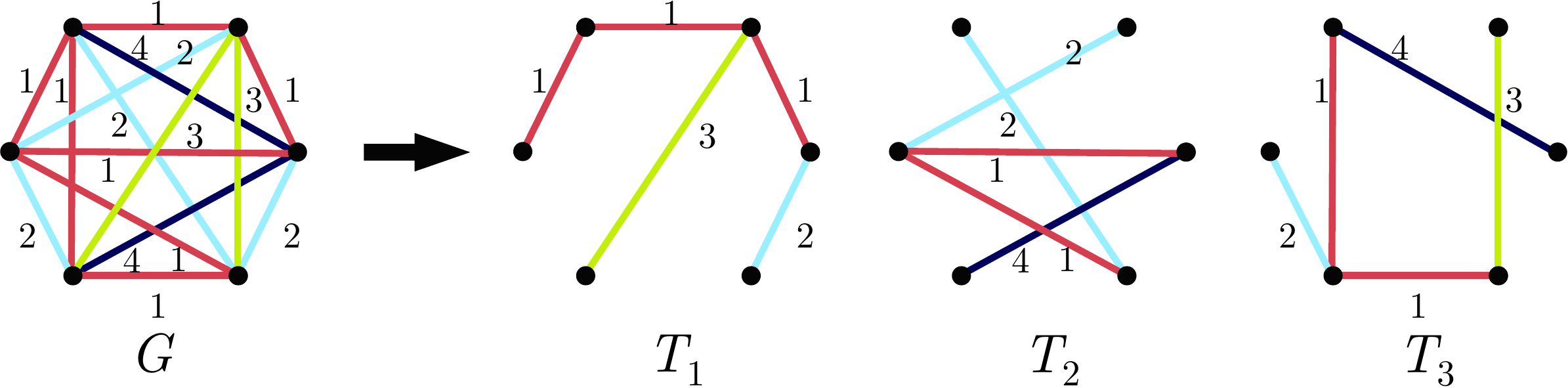}
{An example of Conjecture \ref{conj_180910_1}.}
{fig_180910_1}

By the same argument in the proof of
Remark \ref{rem_180922_1},
we can show that
Conjecture \ref{conj_180910_1} is equivalent
to the following proposition.

\begin{conjecture}
Let $G$ be an edge-colored complete graph
of order $2n$ $(n \ge 3)$.
Let $g$ and $f$ be mappings from $\C$
to $\mathbb{Z}_{\ge 0}$.
If for any color $c \in \C$,
\begin{equation*}
g(c) \le \frac{|E_c(G)|}{n} \le f(c),
\end{equation*}
then $G$ can be partitioned into
$n$ edge-disjoint $(g,f)$-chromatic spanning trees.
\label{conj_180916_1}
\end{conjecture}

%------------------------------------
%   Section 3
\section{Proofs}
\label{Proofs}
%------------------------------------
In this section,
we will prove
Theorem \ref{thm_180815_1},
Theorem \ref{thm_180826_1}, and
Theorem \ref{thm_180830_1}.
In order to prove Theorem \ref{thm_180815_1},
we will use
Lemma \ref{lem_180923_1} and \ref{lem_180923_2},
which will be proved in Section
\ref{proof31} and \ref{proof32}, respectively.
In order to prove Theorem \ref{thm_180826_1},
we will use almost trivial Lemma \ref{lem_180923_3},
which was proved by Suzuki
\cite{Suzuki2013AForests}.

\begin{lemma}
Let $G$ be an edge-colored graph of order $n$.
Let $g$ be a mapping
from $\C$ to $\mathbb{Z}_{\ge 0}$.
$G$ has
a $(g,g)$-chromatic forest
if and only if
\begin{equation*}
\omega(G-E_R(G)) \le
n-\sum_{c \in \C \setminus R}g(c)
\text{~~~ for any } R \subseteq \C.
\end{equation*}
\label{lem_180923_1}
\end{lemma}

Note that,
this lemma requires
the forest neither
to be a spanning forest nor
to have a fixed number of components.

\begin{lemma}
Let $G$ be an edge-colored graph of order $n$.
Let $g$ and $f$ be mappings from $\C$
to $\mathbb{Z}_{\ge 0}$
such that $g(c) \le f(c)$ for any $c \in \C$.
Let $m$ be a positive integer.
$G$ has
a $(g,f)$-chromatic spanning forest
with exactly $m$ components
if and only if
$G$ has both
an $f$-chromatic spanning forest
of size at least
$\sum_{c \in \C}g(c)$
with exactly $m$ components,
and
a $(g,g)$-chromatic forest.
\label{lem_180923_2}
\end{lemma}

Note that,
the $f$-chromatic spanning forest
and the $(g,g)$-chromatic forest
may be different in Lemma \ref{lem_180923_2}.

\begin{lemma}
\begin{equation*}
|E(G)| \le \binom{|V(G)|-\omega(G)+1}{2}
\text{~~ for any graph } G.
\end{equation*}
\label{lem_180923_3}
\end{lemma}

%------------------------------------
\subsection{Proof of Lemma \ref{lem_180923_1}}
\label{proof31}
%------------------------------------
Let $G$ be an edge-colored graph of order $n$.
Let $g$ be a mapping
from $\C$ to $\mathbb{Z}_{\ge 0}$.

First,
we prove the necessity.
Suppose that
$G$ has a $(g,g)$-chromatic forest $F$.
By Definition \ref{def_180815_1},
$|E_c(F)|=g(c)$ for any color $c$.
For any $R \subseteq \C$,
the graph $(V(G),E_{\C \setminus R}(F))$
is a spanning forest of $G-E_R(G)$.
Thus,
\begin{align*}
\omega(G-E_R(G))
&\le \omega((V(G),E_{\C \setminus R}(F)))\\
&=|V(G)|-|E_{\C \setminus R}(F)|\\
&= |V(G)|-\sum_{c \in \C \setminus R}|E_c(F)|\\
&= n-\sum_{c \in \C \setminus R}g(c).
\end{align*}

Next,
we prove the sufficiency.
Suppose that
\begin{equation*}
\omega(G-E_R(G)) \le
n-\sum_{c \in \C \setminus R}g(c)
\text{~~~ for any } R \subseteq \C.
\end{equation*}
Set $m=n-\sum_{c \in \C}g(c)$.
Then,
\begin{equation*}
n-m=\sum_{c \in \C}g(c)
= \sum_{c \in R}g(c)+\sum_{c \in \C \setminus R}g(c)
\text{~~~ for any } R \subseteq \C,
\end{equation*}
that is,
\begin{equation*}
n-\sum_{c \in \C \setminus R}g(c)
= m+\sum_{c \in R}g(c)
\text{~~~ for any } R \subseteq \C.
\end{equation*}
Thus, we have
\begin{equation*}
\omega(G-E_R(G)) \le
m+\sum_{c \in R}g(c)
\text{~~~ for any } R \subseteq \C.
\end{equation*}
Hence, by Theorem \ref{thm_Suzuki2013AForests_1},
$G$ has a $g$-chromatic spanning forest $F$
with exactly $m$ components.
By Definition \ref{def_Suzuki2013AForests},
$|E_c(F)| \le g(c)$ for any color $c \in \C$.
On the other hand,
we have
\begin{equation*}
\sum_{c \in \C}|E_c(F)|
= |E(F)|
= n-m
= n-(n-\sum_{c \in \C}g(c))
= \sum_{c \in \C}g(c).
\end{equation*}
Thus, $|E_c(F)| = g(c)$ for any color $c \in \C$.
Therefore,
by Definition \ref{def_180815_1},
$F$ is a $(g,g)$-chromatic forest of $G$.

%------------------------------------
\subsection{Proof of Lemma \ref{lem_180923_2}}
\label{proof32}
%------------------------------------
Let $G$ be an edge-colored graph of order $n$.
Let $g$ and $f$ be mappings from $\C$
to $\mathbb{Z}_{\ge 0}$
such that $g(c) \le f(c)$ for any $c \in \C$.
Let $m$ be a positive integer.

First,
we prove the necessity.
Suppose that
$G$ has
a $(g,f)$-chromatic spanning forest $F$
with exactly $m$ components.
By Definition \ref{def_180815_1},
$g(c) \le |E_c(F)|$ for any color $c \in \C$.
Thus,
$\sum_{c \in \C}g(c)
\le \sum_{c \in \C}|E_c(F)|
= |E(F)|$.
Hence,
$F$ is an $f$-chromatic spanning forest
of size at least
$\sum_{c \in \C}g(c)$
with exactly $m$ components of $G$.
Since $F$ is a $(g,f)$-chromatic forest,
$F$ contains some $(g,g)$-chromatic forest,
which is also a $(g,g)$-chromatic forest in $G$.

Next,
we prove the sufficiency.
Suppose that
$G$ has both
an $f$-chromatic spanning forest
of size at least
$\sum_{c \in \C}g(c)$
with exactly $m$ components,
and
a $(g,g)$-chromatic forest $F_g$.
Let $F_f$ be an $f$-chromatic spanning forest
of size at least
$\sum_{c \in \C}g(c)$
with exactly $m$ components of $G$
such that
it has the maximum number of edges of $F_g$.

We will prove that
$F_f$ is the desired $(g,f)$-chromatic spanning forest
with exactly $m$ components of $G$
by contradiction.

Suppose that
$F_f$ is not a $(g,f)$-chromatic spanning forest
with exactly $m$ components of $G$.
Then, since $F_f$ is $f$-chromatic
but not $(g,f)$-chromatic,
we may assume that for some color, say color $1$,
$|E_1(F_f)| \le g(1)-1$.

Since $F_g$ is $(g,g)$-chromatic,
$|E_1(F_g)|=g(1)$.
Thus, $|E_1(F_f)| < |E_1(F_g)|$.
Hence, $E_1(F_g) \setminus E_1(F_f) \ne \emptyset$.
Let $e$ be an edge in $E_1(F_g) \setminus E_1(F_f)$.
Adding the edge $e$ to $F_f$,
we consider the resulting graph
$(V(F_f),E(F_f) \cup \{ e \})$
denoted by $F_f^+$.
Since $F_f$ is $f$-chromatic and $e \notin E_1(F_f)$,
we have
\begin{equation*}
|E_c(F_f^+)|=
\begin{cases}
|E_c(F_f)|+1 \le g(c) \le f(c)
& \text{if $c=1$},\\
|E_c(F_f)| \le f(c)
& \text{if $c \ne 1$}.
\end{cases}
\end{equation*}
Thus, $F_f^+$ is also an $f$-chromatic
spanning subgraph of $G$.

If the edge $e$ connects two distinct components
of $F_f$ in $F_f^+$,
then $F_f^+$ is an $f$-chromatic spanning forest
with exactly $m-1$ components of $G$.
Since $F_g$ is $(g,g)$-chromatic,
$|E(F_g)| = \sum_{c \in \C}g(c)$.
Since $|E(F_f)| \ge \sum_{c \in \C}g(c)$,
we have
\begin{equation*}
|E(F_f^+)|=|E(F_f)|+1 \ge \sum_{c \in \C}g(c) +1
= |E(F_g)|+1 > |E(F_g)|.
\end{equation*}
Thus,
$E(F_f^+) \setminus E(F_g) \ne \emptyset$.
Let $e'$ be an edge in $E(F_f^+) \setminus E(F_g)$.
Then, we have
\begin{gather*}
\omega(F_f^+ - e')=\omega(F_f^+)+1=m,\\
|E(F_f^+ - e')|=|E(F_f^+)|-1
=|E(F_f)| \ge \sum_{c \in \C}g(c),
\end{gather*}
where $F_f^+ - e'$ denotes the graph
$(V(F_f^+),E(F_f^+) \setminus \{ e' \})$.
Hence,
since $F_f^+$ is an $f$-chromatic spanning forest
of $G$,
$F_f^+ - e'$ is an $f$-chromatic spanning forest
of size at least
$\sum_{c \in \C}g(c)$
with exactly $m$ components of $G$.
Recall that
$e \in E(F_g)$ and $e' \notin E(F_g)$.
Then,
$F_f^+ - e'$, namely,
$(V(F_f),(E(F_f) \cup \{ e \}) \setminus \{ e' \})$
has more edges of $F_g$ than $F_f$,
which is a contradiction
to the maximality of $F_f$.

Therefore, we may assume that
the both endpoints of $e$
are contained in one component of $F_f$.
Then,
$\omega(F_f^+)=\omega(F_f)=m$ and
$F_f^+$ has exactly one cycle $C$,
which contains $e$.
Since $F_g$ has no cycles,
$C$ has some edge $e' \notin E(F_g)$.
Then, $F_f^+ - e'$ is a forest and
\begin{gather*}
\omega(F_f^+ - e')=\omega(F_f^+)=m,\\
|E(F_f^+ - e')|=|E(F_f^+)|-1
=|E(F_f)| \ge \sum_{c \in \C}g(c).
\end{gather*}
Thus,
since $F_f^+$ is an $f$-chromatic spanning subgraph
of $G$,
$F_f^+ - e'$ is an $f$-chromatic spanning forest
of size at least
$\sum_{c \in \C}g(c)$
with exactly $m$ components of $G$.
Recall that
$e \in E(F_g)$ and $e' \notin E(F_g)$.
Then,
$F_f^+ - e'$, namely,
$(V(F_f),(E(F_f) \cup \{ e \}) \setminus \{ e' \})$
has more edges of $F_g$ than $F_f$,
which is a contradiction
to the maximality of $F_f$.

Consequently,
$F_f$ is the desired $(g,f)$-chromatic spanning forest
with exactly $m$ components of $G$.

%------------------------------------
\subsection{Proof of Theorem \ref{thm_180815_1}}
\label{proof22}
%------------------------------------
Let $G$ be an edge-colored graph of order $n$.
Let $g$ and $f$ be mappings from $\C$
to $\mathbb{Z}_{\ge 0}$
such that $g(c) \le f(c)$ for any $c \in \C$.
Let $m$ be a positive integer
such that $n \ge m + \sum_{c \in \C}g(c)$.

First,
we prove the necessity.
Suppose that
$G$ has
a $(g,f)$-chromatic spanning forest $F$
with exactly $m$ components.
Since $F$ is a $(g,f)$-chromatic forest,
$F$ contains some $(g,g)$-chromatic forest.
Thus, by Lemma \ref{lem_180923_1}, we have
\begin{equation*}
\omega(G-E_R(G)) \le
n-\sum_{c \in \C \setminus R}g(c)
\text{~~~ for any } R \subseteq \C.
\end{equation*}
On the other hand,
since $F$ is a $(g,f)$-chromatic spanning forest
with exactly $m$ components of $G$,
$F$ is an $f$-chromatic spanning forest
with exactly $m$ components of $G$.
Thus, by Theorem \ref{thm_Suzuki2013AForests_1},
we have
\begin{equation*}
\omega(G-E_R(G)) \le m+\sum_{c \in R}f(c)
\text{~~~~~ for any } R \subseteq \C.
\end{equation*}
Therefore,
\begin{equation*}
\omega(G-E_R(G)) \le
\min
\{~
m+\sum_{c \in R}f(c),~~
n-\sum_{c \in \C \setminus R}g(c)
~\}
\text{~~~ for any } R \subseteq \C.
\end{equation*}

Next,
we prove the sufficiency.
Suppose that
\begin{equation}
\omega(G-E_R(G)) \le
\min
\{~
m+\sum_{c \in R}f(c),~~
n-\sum_{c \in \C \setminus R}g(c)
~\}
\text{~~~ for any } R \subseteq \C.
\label{eq_180925_1}
\end{equation}

By (\ref{eq_180925_1}), we have
\begin{equation*}
\omega(G-E_R(G)) \le m+\sum_{c \in R}f(c)
\text{~~~ for any } R \subseteq \C.
\end{equation*}
Thus, by Theorem \ref{thm_Suzuki2013AForests_1},
$G$ has an $f$-chromatic spanning forest $F$
with exactly $m$ components of $G$.
By our assumption
that $n \ge m + \sum_{c \in \C}g(c)$,
we have
\begin{equation*}
|E(F)| = n-m \ge \sum_{c \in \C}g(c).
\end{equation*}
Thus, $F$ is
an $f$-chromatic spanning forest
of size at least
$\sum_{c \in \C}g(c)$
with exactly $m$ components.

On the other hand,
by (\ref{eq_180925_1}), we have
\begin{equation*}
\omega(G-E_R(G)) \le n-\sum_{c \in \C \setminus R}g(c)
\text{~~~ for any } R \subseteq \C.
\end{equation*}
Thus, by Lemma \ref{lem_180923_1},
$G$ has a $(g,g)$-chromatic forest.

Therefore,
by Lemma \ref{lem_180923_2},
$G$ has
a $(g,f)$-chromatic spanning forest
with exactly $m$ components.

%------------------------------------
\subsection{Proof of Theorem \ref{thm_180826_1}}
\label{proof23}
%------------------------------------
Let $G$ be an edge-colored graph of order $n$.
Let $g$ and $f$ be mappings from $\C$
to $\mathbb{Z}_{\ge 0}$.
Let $m$ be a positive integer such that $n \ge m$.
Suppose that $|E(G)|>\binom{n-m}{2}$ and
\begin{equation}
g(c) \le \frac{|E_c(G)|}{|E(G)|}(n-m) \le f(c)
\text{~~~~~ for any color } c \in \C.
\label{eq_180926_1}
\end{equation}
Then, since $\sum_{c \in \C}|E_c(G)|=|E(G)|$,
we have
\begin{equation}
\sum_{c \in \C}g(c)
\le
\sum_{c \in \C}\frac{|E_c(G)|}{|E(G)|}(n-m)
= n-m,
\text{ that is, }
n \ge m + \sum_{c \in \C}g(c).
\label{eq_180926_2}
\end{equation}

We will prove that
$G$ has a $(g,f)$-chromatic spanning forest
with exactly $m$ components
by contradiction.

Suppose that
$G$ has no $(g,f)$-chromatic spanning forests
with exactly $m$ components.
By (\ref{eq_180926_2}) and our assumption,
we can apply Theorem \ref{thm_180815_1} to $G$
and we have
\begin{equation*}
\omega(G-E_R(G)) >
\min
\{~
m+\sum_{c \in R}f(c),~~
n-\sum_{c \in \C \setminus R}g(c)
~\}
\text{~~~ for some } R \subseteq \C.
\end{equation*}

That is,
$\omega(G-E_R(G)) \ge m+\sum_{c \in R}f(c) +1$ or
$\omega(G-E_R(G)) \ge
n-\sum_{c \in \C \setminus R}g(c) +1$
for some $R \subseteq \C$.
We denote $G-E_R(G)$ by $G'$.

\begin{claim}
\begin{equation*}
\omega(G') \ge m+1
\text{~ and ~}
\omega(G') \ge n+1- \frac{|E(G')|}{|E(G)|}(n-m)
\end{equation*}
\label{clm_180926_1}
\end{claim}
\begin{proof}
First, we suppose that
$\omega(G') \ge m+\sum_{c \in R}f(c)+1$
for some $R \subseteq \C$.
Since $f(c) \ge 0$ for any color $c$,
$\omega(G') \ge m+\sum_{c \in R}f(c)+1 \ge m+1$.

By our assumption (\ref{eq_180926_1}),
\begin{align*}
\sum_{c \in R}f(c)
&\ge
\sum_{c \in R}\frac{|E_c(G)|}{|E(G)|}(n-m)
= \frac{n-m}{|E(G)|}\sum_{c \in R}|E_c(G)|
= \frac{n-m}{|E(G)|}|E_R(G)|\\
&= \frac{n-m}{|E(G)|}(|E(G)|-|E(G')|)
= n-m-\frac{|E(G')|}{|E(G)|}(n-m).
\end{align*}
Thus, we have
\begin{align*}
\omega(G')
&\ge m+\sum_{c \in R}f(c)+1\\
&\ge m+ n-m-\frac{|E(G')|}{|E(G)|}(n-m) +1
= n+1- \frac{|E(G')|}{|E(G)|}(n-m).
\end{align*}

Next, we suppose that
$\omega(G') \ge
n-\sum_{c \in \C \setminus R}g(c) +1$
for some $R \subseteq \C$.
By (\ref{eq_180926_2}),
$\sum_{c \in \C}g(c) \le n-m$.
Thus, we have
\begin{align*}
\omega(G')
\ge n-\sum_{c \in \C \setminus R}g(c) +1
\ge n-\sum_{c \in \C}g(c) +1
\ge n-(n-m)+1
= m+1.
\end{align*}

By our assumption (\ref{eq_180926_1}),
\begin{equation*}
\sum_{c \in \C \setminus R}g(c)
\le
\sum_{c \in \C \setminus R}\frac{|E_c(G)|}{|E(G)|}(n-m)
= \frac{|E_{\C \setminus R}(G)|}{|E(G)|}(n-m)
= \frac{|E(G')|}{|E(G)|}(n-m).
\end{equation*}
Thus, we have
\begin{align*}
\omega(G')
&\ge n-\sum_{c \in \C \setminus R}g(c) +1\\
&\ge n-\frac{|E(G')|}{|E(G)|}(n-m) +1
= n+1- \frac{|E(G')|}{|E(G)|}(n-m).
\end{align*}
\end{proof}

By Claim \ref{clm_180926_1},
\begin{equation*}
n-\omega(G')+1 \le \frac{|E(G')|}{|E(G)|}(n-m).
\end{equation*}
Since $n \ge \omega(G')$,
$n-\omega(G')+1 \ge 1$,
that is, $n-\omega(G')+1 \ne 0$.
Thus,
\begin{align*}
|E(G)|
&\le
\frac{n-m}{n-\omega(G')+1}|E(G')|.
\intertext{
Since $|V(G')|=|V(G)|=n$,
by Lemma \ref{lem_180923_3},}
|E(G)|
&\le
\frac{n-m}{n-\omega(G')+1}
\binom{|V(G')|-\omega(G')+1}{2}\\
&\le
\frac{n-m}{n-\omega(G')+1}
\times \frac{(n-\omega(G')+1)(n-\omega(G'))}{2}\\
&=
\frac{(n-m)(n-\omega(G'))}{2}.
\intertext{
By Claim \ref{clm_180926_1},
$\omega(G') \ge m+1$.
Thus,}
|E(G)|
&\le
\frac{(n-m)(n-(m+1))}{2}
= \binom{n-m}{2},
\end{align*}
which contradicts to our assumption
that $|E(G)|>\binom{n-m}{2}$.

Therefore,
$G$ has a $(g,f)$-chromatic spanning forest
with exactly $m$ components.

%------------------------------------
\subsection{Proof of Theorem \ref{thm_180830_1}}
\label{proof26}
%------------------------------------
If an edge-colored complete graph $G$ of order $n$
has a spanning tree $T$
with the same color probability distribution
as that of $G$,
that is,
\begin{equation*}
\frac{|E_c(G)|}{|E(G)|} = \frac{|E_c(T)|}{|E(T)|}
\text{~~ for any color } c \in \C.
\end{equation*}
then
\begin{equation*}
|E_c(G)| = \frac{|E_c(T)||E(G)|}{|E(T)|}
= \frac{|E_c(T)|n(n-1)/2}{n-1}
= \frac{|E_c(T)|n}{2}
\text{~~ for any color } c \in \C.
\end{equation*}
Thus, since $|E_c(T)|$ is an integer,
$|E_c(G)|$ is an integral multiple of $n/2$.

Next, let $G$ be
an edge-colored complete graph $G$ of order $n$.
For any color $c \in \C$,
we suppose that
$|E_c(G)| = k_c \times n/2$
for some $k_c \in \mathbb{Z}_{\ge 0}$.
By Theorem \ref{thm_180913_1},
$G$ has a spanning tree $T$
such that
\begin{equation*}
\left\lfloor
\frac{|E_c(G)|}{|E(G)|}(n-1)
\right\rfloor
\le
|E_c(T)|
\le
\left\lceil
\frac{|E_c(G)|}{|E(G)|}(n-1)
\right\rceil
\text{~~ for any color } c \in \C.
\end{equation*}
Since $|E(G)|=n(n-1)/2$ and
$|E_c(G)| = k_c \times n/2$
($k_c \in \mathbb{Z}_{\ge 0}$),
we have
\begin{equation*}
k_c
=
\lfloor k_c \rfloor
=
\left\lfloor
\frac{|E_c(G)|}{|E(G)|}(n-1)
\right\rfloor
\le
|E_c(T)|
\le
\left\lceil
\frac{|E_c(G)|}{|E(G)|}(n-1)
\right\rceil
=
\lceil k_c \rceil
=
k_c.
\end{equation*}
Thus, $|E_c(T)|=k_c$.
Then,
\begin{equation*}
\frac{|E_c(G)|}{|E(G)|}
= \frac{k_c \times n/2}{n(n-1)/2}
= \frac{k_c}{n-1}
= \frac{|E_c(T)|}{|E(T)|}
\text{~~ for any color } c \in \C.
\end{equation*}
Therefore,
the color probability distribution of $T$
is the same as that of $G$.

%------------------------------------
%   Acknowledgements
%------------------------------------
\bigskip
\noindent
\textbf{Acknowledgments}
This work was supported by
JSPS KAKENHI Grant Number 16K05254.
%We appreciate the referees for helpful comments.

%------------------------------------
%   References
%------------------------------------
\bibliography{Mendeley}

\begin{thebibliography}{10}

\bibitem{Akbari2007MulticoloredGraphs}
Akbari,~S. and Alipour,~A., {Multicolored trees in complete graphs}, {\em
  Journal of Graph Theory}, {\bfseries 54} (2007), 221--232.

\bibitem{Brualdi1996MulticoloredGraphs}
Brualdi,~Richard~A. and Hollingsworth,~Susan, {Multicolored trees in complete
  graphs}, {\em Journal of Combinatorial Theory, Series B}, {\bfseries 68}
  (1996), 310--313.

\bibitem{Carraher2016Edge-disjointGraphs}
Carraher,~James~M.; Hartke,~Stephen~G.; and Horn,~Paul, {Edge-disjoint rainbow
  spanning trees in complete graphs}, {\em European Journal of Combinatorics},
  {\bfseries 57} (2016), 71--84.

\bibitem{Chartrand2008RAINBOWGRAPHS}
Chartrand,~Gary; Johns,~Garry~L.; McKeon,~Kathleen~A.; and Zhang,~Ping,
  {Rainbow connection in graphs}, {\em Mathematica Bohemica}, {\bfseries 133}
  (2008), 85--98.

\bibitem{Erdos1975Anti-RamseyTheorems}
Erd{\H{o}}s,~P.; Simonovits,~M.; and S{\'{o}}s,~V.~T., {Anti-Ramsey theorems},
  {\em Infinite and finite sets (Colloq., Keszthely, 1973; dedicated to P.
  Erd{\H{o}}s on his 60th birthday), Vol. II, Colloq. Math. Soc. J{\'{a}}nos
  Bolyai, Vol. 10}, {\em North-Holland, Amsterdam},  (1975), 633--643.

\bibitem{Fujita2010RainbowSurvey}
Fujita,~Shinya; Magnant,~Colton; and Ozeki,~Kenta, {Rainbow generalizations of
  Ramsey theory: a survey}, {\em Graphs and Combinatorics}, {\bfseries 26}
  (2010), 1--30.

\bibitem{Horn2018RainbowOne-factorizations}
Horn,~Paul, {Rainbow spanning trees in complete graphs colored by
  one-factorizations}, {\em Journal of Graph Theory}, {\bfseries 87} (2018),
  333--346.

\bibitem{Kaneko2006TwoGraphs}
Kaneko,~Atsushi; Kano,~Mikio; and Suzuki,~Kazuhiro, {Two edge-disjoint
  heterochromatic spanning trees in colored complete graphs}, {\em
  Matimy{\'{a}}s Matematika}, {\bfseries 29} (2006), 49--51.

\bibitem{Kano2008MonochromaticSurvey}
Kano,~Mikio and Li,~Xueliang, {Monochromatic and heterochromatic subgraphs in
  edge-colored graphs - a survey}, {\em Graphs and Combinatorics}, {\bfseries
  24} (2008), 237--263.

\bibitem{Krussel2000SpanningK2n}
Krussel,~John; Marshall,~Susan; and Verrall,~Helen, {Spanning trees orthogonal
  to one-factorizations of K2n}, {\em Ars Combinatoria}, {\bfseries 57} (2000),
  77--82.

\bibitem{Li2013RainbowSurvey}
Li,~Xueliang; Shi,~Yongtang; and Sun,~Yuefang, {Rainbow connections of graphs:
  a survey}, {\em Graphs and Combinatorics}, {\bfseries 29} (2013), 1--38.

\bibitem{Suzuki2006AGraph}
Suzuki,~Kazuhiro, {A necessary and sufficient condition for the existence of a
  heterochromatic spanning tree in a graph}, {\em Graphs and Combinatorics},
  {\bfseries 22} (2006), 261--269.

\bibitem{Suzuki2013AForests}
Suzuki,~Kazuhiro, {A generalization of heterochromatic graphs and f-chromatic
  spanning forests}, {\em Graphs and Combinatorics}, {\bfseries 29} (2013),
  715--727.

\end{thebibliography}
\bibliographystyle{tutetuti}

%------------------------------------
\end{document}